\documentclass[twoside,a4paper,10pt]{amsart}

\usepackage{fullpage}
\usepackage{amscd}
\usepackage{amsmath}
\usepackage{amssymb}
\usepackage{amsthm}
\usepackage{color}
\usepackage{hyperref}
\usepackage[noabbrev,capitalize]{cleveref}
\usepackage{mathrsfs}
\usepackage{mathtools}
\usepackage{pict2e}
\usepackage{stackrel}
\usepackage{tikz}
\usepackage{wasysym}
\usepackage[all]{xy}
\usepackage{xcolor}

\hypersetup{
    colorlinks,
    linkcolor={red!50!black},
    citecolor={blue!50!black},
    urlcolor={blue!80!black}
}

\usepackage{palatino}
\allowdisplaybreaks

\newtheorem{lemma}{Lemma}[section]
\newtheorem{remark}[lemma]{Remark}
\newtheorem{example}[lemma]{Example}
\newtheorem{theorem}[lemma]{Theorem}
\newtheorem{corollary}[lemma]{Corollary}
\newtheorem{definition}[lemma]{Definition}
\newtheorem{proposition}[lemma]{Proposition}

\author{Daniel Robert-Nicoud}
\date{}
\title{A model structure for the Goldman--Millson theorem}
\address{Laboratoire Analyse, G\'eom\'etrie et Applications, Universit\'e Paris 13, Sorbonne Paris Cit\'e, 99 Avenue Jean Baptiste Cl\'ement, 93430 Villetaneuse, France}
\email{robert-nicoud@math.univ-paris13.fr}

\newcommand{\ad}{\ensuremath{\mathrm{ad}}}
\newcommand{\antishriek}{\mbox{\footnotesize{\rotatebox[origin=c]{180}{$!$}}}}

\renewcommand{\bar}{\ensuremath{\mathrm{B}}}

\newcommand{\C}{\ensuremath{\mathscr{C}}}
\newcommand{\cat}{\ensuremath{\mathcal{C}}}

\newcommand{\cdgl}{\ensuremath{\mathsf{cdgLie}}}

\newcommand{\colie}{\ensuremath{\mathsf{coLie}}}
\newcommand{\colim}[1]{\ensuremath{\underset{#1}{\mathrm{colim}}}}
\newcommand{\com}{\ensuremath{\mathsf{Com}}}

\newcommand{\dgl}{\ensuremath{\mathsf{dgLie}}}

\newcommand{\fdvect}{\ensuremath{\mathsf{fdVect}}}

\newcommand{\fsets}{\ensuremath{\mathsf{fSets}}}
\newcommand{\fqi}{\mathrm{Fqi}}

\newcommand{\g}{\ensuremath{\mathfrak{g}}}

\newcommand{\h}{\ensuremath{\mathfrak{h}}}
\newcommand{\ho}[1]{\ensuremath{\mathrm{Ho}\left(#1\right)}}
\newcommand{\id}{\ensuremath{\mathrm{id}}}
\newcommand{\ind}[1]{\ensuremath{\mathsf{ind}(#1)}}
\newcommand{\lie}{\ensuremath{\mathsf{Lie}}}

\newcommand{\mc}{\ensuremath{\mathrm{mc}}}
\newcommand{\MC}{\ensuremath{\mathrm{MC}}}
\newcommand{\MCbar}{\ensuremath{\overline{\mathrm{MC}}}}

\renewcommand{\k}{\ensuremath{\mathbb{K}}}
\renewcommand{\L}{\ensuremath{\mathscr{L}}}
\renewcommand{\P}{\ensuremath{\mathscr{P}}}

\newcommand{\pro}[1]{\ensuremath{\mathsf{pro}(#1)}}

\renewcommand{\S}{\ensuremath{\mathbb{S}}}
\newcommand{\sets}{\ensuremath{\mathsf{Sets}}}

\newcommand{\ssets}{\ensuremath{\mathsf{sSet}}}

\newcommand{\susp}{\ensuremath{\mathscr{S}}}

\newcommand{\vect}{\ensuremath{\mathsf{Vect}}}

\makeatletter
\newcommand{\adjunction}{\@ifstar\named@adjunction\normal@adjunction}
\newcommand{\normal@adjunction}[4]{%
  #1\ \colon #2%
  \mathrel{\vcenter{%
    \offinterlineskip\m@th
    \ialign{%
      \hfil$##$\hfil\cr
      \longrightharpoonup\cr
      \noalign{\kern-.3ex}
      \smallbot\cr
      \longleftharpoondown\cr
    }%
  }}%
  #3 \noloc\ #4%
}
\newcommand{\named@adjunction}[4]{%
  #2%
  \mathrel{\vcenter{%
    \offinterlineskip\m@th
    \ialign{%
      \hfil$##$\hfil\cr
      \scriptstyle#1\cr
      \noalign{\kern.1ex}
      \longrightharpoonup\cr
      \noalign{\kern-.3ex}
      \smallbot\cr
      \longleftharpoondown\cr
      \scriptstyle#4\cr
    }%
  }}%
  #3%
}
\newcommand{\longrightharpoonup}{\relbar\joinrel\rightharpoonup}
\newcommand{\longleftharpoondown}{\leftharpoondown\joinrel\relbar}
\newcommand\noloc{%
  \nobreak
  \mspace{6mu plus 1mu}
  {:}
  \nonscript\mkern-\thinmuskip
  \mathpunct{}
  \mspace{2mu}
}
\newcommand{\smallbot}{%
  \begingroup\setlength\unitlength{.15em}%
  \begin{picture}(1,1)
  \roundcap
  \polyline(0,0)(1,0)
  \polyline(0.5,0)(0.5,1)
  \end{picture}%
  \endgroup
}
\makeatother



\thanks{The author was supported by grants from R\'egion Ile-de-France, and the grant ANR-14-CE25-0008-01 project SAT}

\begin{document}

\begin{abstract}
	By a result of Vallette \cite{vallette14}, we put a sensible model structure on the category of conilpotent Lie coalgebras. This gives us a powerful tool to study the subcategory of Lie algebras obtained by linear dualization, also known as the category of pronilpotent Lie algebras. This way, we recover weaker versions of the celebrated Goldman--Millson theorem and Dolgushev--Rogers theorem by purely homotopical methods. We explore the relations of this procedure with the existent literature, namely the works of Lazarev--Markl and Buijs--F\'elix--Murillo--Tanr\'e.
\end{abstract}

\maketitle

\setcounter{tocdepth}{1}

\tableofcontents

\section{Introduction}

The celebrated Goldman--Millson theorem \cite{goldman88} in one of its more recent versions states that every filtered quasi-isomorphism of differential graded Lie algebras --- i.e. quasi-isomorphisms that are compatible with certain filtrations of the algebras --- induces an isomorphism between the moduli spaces of Maurer--Cartan elements of the algebras. Recently, this theorem was extended by V. A. Dolgushev and C. L. Rogers \cite{dolgushev15}, who proved that such a morphism induces a weak equivalence of simplicial sets between the whole deformation $\infty$-groupoids of the algebras, of which the moduli space of Maurer--Cartan elements is the zeroth homotopy group. These theorems have a distinct homotopical flavor --- after all, they say that the class of filtered quasi-isomorphisms is sent into the classe of weak equivalences of simplicial sets under a certain functor. However, the proofs of these theorems are not done by purely homotopical methods, but rather by working explicitly with the algebras and performing some induction using the filtrations. There is a reason behind this: the filtered quasi-isomorphisms are not very well behaved and do not form the class of weak equivalences of any model structure on the category of differential graded Lie algebras that we know of, even after closing them by the 2-out-of-3 property.

\medskip

The primary goal of this paper is to provide a fully homotopical and self-contained approach to the proof of the Goldman--Millson theorem and to the Dolgushev--Rogers theorem. We begin by considering the model structure on Lie coalgebras given by B. Vallette in \cite{vallette14}, of which we completely characterize the weak equivalences. Thanks to some results of \cite{legrignou16}, we show that the category of conilpotent Lie coalgebras is equivalent to the category of pro-objects in finite dimensional, nilpotent Lie algebras, and thus we obtain a model structure on this category too. An interplay between this model structure and the limit functor then allows us to show that we can see the gauge relation of Maurer--Cartan elements as a homotopy relation between certain morphisms representing Maurer--Cartan elements. A model categorical argument then immediately gives us a slightly weaker version of the Goldman--Millson theorem. Further, using simplicial framings, we extend the argument to prove a version of the Dolgushev--Rogers theorem.

\medskip

This work has close links with recent works by other authors, which will be explored throughout the text. Starting from the general theory of \cite{vallette14}, which gives a model structure on the category of conilpotent coalgebras over a cooperad, we recover the model structure on the category of pro-objects in finite-dimensional, nilpotent Lie algebras described by A. Lazarev and M. Markl in \cite{lazarev13}. The idea of a homotopical approach to the proof of the Goldman--Millson and the Dolgushev--Rogers theorems is already present in the work of U. Buijs, Y. F\'elix, A. Murillo, and D. Tanr\'e --- more specifically in \cite{buijs16}, where they obtain a result which is strictly stronger than our \cref{thm:homotopical GM,thm:homotopical DR}. However, they need to use a version of the Dolgushev--Rogers theorem in their proof, while our approach is more self-contained, if not as powerful. Moreover, the techniques we present in \cref{sect:DRThm} can be applied to the results of Buijs--F\'elix--Murillo--Tanr\'e to give a modest generalization.

\medskip

An objective of this paper is also to understand the links between all these results. In order to do this, we find ourselves confronted with various notions coming from different fields. In order to make the text accessible to a wide variety of mathematicians, we have interspersed it with various reminders on certain technical notions.

\subsection*{Structure of the paper}

In \cref{sect:ModelCat}, we recall some notions of category theory and model categories that we will need throughout the text. In \cref{sect:GMandDRThms},  we give a rapid overview of some very important results about the Maurer--Cartan elements in Lie algebras: the Goldman--Millson theorem and the Dolgushev--Rogers theorem. In \cref{sect:ValletteModelStr}, we review the Vallette model structure on the category of conilpotent coalgebras over a cooperad. This section is mostly operadic, and we assume a basic familiarity with the language and basic results of operad theory throughout it. However, it can be skipped and the reader uninterested in this part of the theory can find the case of Lie coalgebras --- which we need in the rest of the article --- summarized in \cref{subsect:coLieCog}. Going on, in \cref{sect:GMThm}, we use the theory developed in the previous parts of the paper to derive a partial version of the Goldman--Millson theorem using purely homotopy theoretical methods. Then, in \cref{sect:DRThm}, we do the same with the Dolgushev--Rogers theorem, which can be seen as a higher analogue of the Goldman--Millson theorem.

\subsection*{Acknowledgements}

I am grateful to Ricardo Campos for the time he dedicated discussing with me during the final phases of the writing of this article, as well --- of course --- to Bruno Vallette for the uncountable answers, comments and always constructive critiques.

\subsection*{Notation and conventions}

Throughout the paper, we work over a field $\k$ of characteristic $0$.

\medskip

Unless stated otherwise, we always work over cochain complexes over $\k$: all operads and all algebras are differential graded, and their differentials have degree $1$. In particular, Maurer--Cartan elements of dg Lie algebras have degree $1$.

\medskip

For operads, cooperads and related concepts we follow as closely as possible the notations of the book \cite{vallette12}.

\section{Some notions of category theory} \label{sect:ModelCat}

In this section, we give a reminder of the basic categorical notions we will need later, such as equivalences of categories, ind- and pro-objects in a category (i.e. formal colimits and limits), model categories and the homotopy category of a model category.

\subsection{Equivalences of categories}

Given two categories $\cat$ and $\mathcal{D}$, an \emph{isomorphism of categories} between them is a functor $F:\cat\to\mathcal{D}$ such that there exists another functor $F^{-1}:\mathcal{D}\to\cat$ satisfying $F^{-1}F = 1_\cat$ and $FF^{-1} = 1_\mathcal{D}$. This notion is far too strict to be really useful. A more sensible notion to compare categories is the following one.

\begin{definition}
	An \emph{equivalence of categories} between $\cat$ and $\mathcal{D}$ is a functor $F:\cat\to\mathcal{D}$ such that there exists a functor $G:\mathcal{D}\to\cat$ and two natural isomorphisms $GF\cong1_\cat$ and $FG\cong1_\mathcal{D}$. An \emph{anti-equivalence of categories} between $\cat$ and $\mathcal{D}$ is an equivalence of categories between $\cat^{op}$ and $\mathcal{D}$.
\end{definition}

There is another notion of equivalence of categories which is at first sight stronger than the previous one.

\begin{definition}
	An \emph{adjoint equivalence of categories} between $\cat$ and $\mathcal{D}$ is an adjoint pair
	\[
	\adjunction{F}{\cat}{\mathcal{D}}{G}
	\]
	such that the unit and counit maps are natural isomorphisms.
\end{definition}

In fact, the two notions of equivalence of categories are the same.

\begin{theorem} \label{thm:equivOfCats}
	Let $F:\cat\to\mathcal{D}$ be a functor. The following statements are equivalent.
	\begin{enumerate}
		\item The functor $F$ is an equivalence of categories.
		\item The functor $F$ is part of an adjoint equivalence of categories.
		\item The functor $F$ is fully faithful and essentially surjective, i.e. for every object $d\in\mathcal{D}$ there exists $c\in\cat$ such that $d\cong F(c)$.
	\end{enumerate}
\end{theorem}

For details on these notions, see for example the book \cite[pp. 92--95]{maclane70}.

\subsection{Categories of ind-objects and categories of pro-objects}

Let $\cat$ be a category. One can then consider the category of ``formal colimits'' in $\cat$ (e.g. \cite[Sect. 8]{sga4.1}).

\begin{definition}
	A (non-empty) small category $\cat$ is \emph{filtered} if
	\begin{enumerate}
		\item for every two objects $x,y\in\cat$, there exists an object $z\in\cat$ and two arrows $z\to x$ and $z\to y$, and
		\item for every two arrows $f,g:b\to a$ in $\cat$, there exists an arrow $h:c\to b$ such that $fh=gh$.
	\end{enumerate}
	The category $\cat$ is \emph{cofiltered} if it satisfies the dual properties.
\end{definition}

\begin{definition}
	The category $\ind{\cat}$ of ind-objects in $\cat$ is the category that has as objects all the diagrams $F:\mathcal{D}\to\cat$ with $\mathcal{D}$ a small filtered category. If $F:\mathcal{D}\to\cat$ and $G:\mathcal{E}\to\cat$ are two objects in $\ind{\cat}$, the set of morphisms between them is
	$$\hom_{\ind{\cat}}(F,G)\coloneqq\lim_{d\in\mathcal{D}}\colim{e\in\mathcal{E}}\hom_\cat(F(d),G(e))\ ,$$
	where the limit and the colimit are taken in $\sets$.
\end{definition}

Dually, one can also consider the category of ``formal limits'' in $\cat$ (e.g. \cite[Sect. A.2]{grothendieck60} and \cite[Sect. 8]{sga4.1}).

\begin{definition}
	The category $\pro{\cat}$ of pro-objects in $\cat$ is the category that has as objects all the diagrams $F:\mathcal{D}\to\cat$ with $\mathcal{D}$ a small cofiltered category. If $F:\mathcal{D}\to\cat$ and ,$G:\mathcal{E}\to\cat$ are two objects in $\pro{\cat}$, the set of morphisms between them is
	$$\hom_{\pro{\cat}}(F,G)\coloneqq\lim_{e\in\mathcal{E}}\colim{d\in\mathcal{D}}\hom_\cat(F(d),G(e))\ ,$$
	where the limit and the colimit are taken in $\sets$.
\end{definition}

\begin{lemma} \label{lemma:antiEquivIndPro}
	Let $\cat'$ and $\cat''$ be two equivalent or anti-equivalent categories. Then $\ind{\cat'}$ is equivalent to $\ind{\cat''}$, respectively anti-equivalent to $\pro{\cat''}$.
\end{lemma}

\begin{proof}
	This is straightforward.
\end{proof}

We recall the definition of a compact object in a category.

\begin{definition}
	Let $\cat$ be a category that admits filtered colimits. An object $c\in\cat$ is said to be \emph{compact} if the functor
	$$\hom_\cat(c,-):\cat\longrightarrow\sets$$
	preserves filtered colimits.
\end{definition}

\begin{proposition} \label{prop:CEquivToIndC}
	Let $\cat$ be a category and let $\cat'$ be a full subcategory of $\cat$. Further, assume that
	\begin{enumerate}
		\item \label{it:CCocomplete} the category $\cat$ is cocomplete,
		\item \label{it:delta} there exists a functor
		$$\delta:\cat\longrightarrow\ind{\cat'}$$
		such that the composite $\colim{}\,\delta$ is naturally isomorphic to the identity functor of $\cat$, where $\colim{}$ is (a choice for) the functor
		$$\colim{}:\ind{\cat'}\longrightarrow\cat$$
		given by taking the colimit in $\cat$, and
		\item \label{it:compactObj} every object in $\cat'$ is compact in $\cat$.
	\end{enumerate}
	Then the functors $\delta$ and $\colim{}$ exhibit an equivalence of categories between $\cat$ and $\ind{\cat'}$.
\end{proposition}

\begin{proof}
	Assumption (\ref{it:CCocomplete}) guarantees the existence of a colimit functor. The natural isomorphism
	$$\colim{}\,\delta\cong\id_\cat$$
	is given by assumption (\ref{it:delta}). We are left to prove that there exists a natural isomorphism
	$$\delta\,\colim{}\cong\id_{\ind{\cat'}}\ .$$
	Let $F:\mathcal{D}\to\cat'$ and $G:\mathcal{E}\to\cat'$ be two objects in $\ind{\cat'}$. Then we have
	\begin{align*}
		\hom_{\ind{\cat'}}(F,G) =&\ \lim_{d\in\mathcal{D}}\colim{e\in\mathcal{E}}\hom_{\cat'}(F(d),G(e))\\
		\cong&\ \lim_{d\in\mathcal{D}}\colim{e\in\mathcal{E}}\hom_\cat(F(d),G(e))\\
		\cong&\ \lim_{d\in\mathcal{D}}\hom_\cat(F(d),\colim{}\,G)\\
		\cong&\ \hom_\cat(\colim{}\,F,\colim{}\,G)\ .
	\end{align*}
	In the second line we used the fact that $\cat'$ is a full subcategory of $\cat$, and in the third one the fact that $F(d)\in\cat'$ is always a compact object by assumption (\ref{it:compactObj}). It follows that for any $F\in\ind{\cat'}$ we have
	\begin{align*}
		\hom_{\ind{\cat'}}(\delta\,\colim{}\,F,F) \cong&\ \hom_\cat(\colim{}\,\delta\,\colim{}\,F,\colim{}\,F)\\
		\cong&\ \hom_\cat(\colim{}\,F,\colim{}\,F)\ ,
	\end{align*}
	where we used the fact that $\colim{}\,\delta\cong\id_\cat$. Then, the identity morphism of $\colim{}\ F$ provides a natural isomorphism $\delta\,\colim{}\cong\id_{\ind{\cat'}}$ as we wanted.
\end{proof}

\begin{example}
	The easiest example of application of \cref{prop:CEquivToIndC} is the following one. Let $\cat\coloneqq\sets$ be the category of (small) sets and let $\cat'\coloneqq\fsets$ be the full subcategory of finite sets. Every set is the colimit of its finite subsets, so taking the diagram of all finite subsets gives a functor
	$$\delta:\sets\longrightarrow\ind{\fsets}\ .$$
	As $\sets$ is cocomplete and finite sets are compact objects in $\sets$, \cref{prop:CEquivToIndC} tells us that $\sets$ is equivalent to $\ind{\fsets}$.
\end{example}

\begin{example}
	Another classical example is as follows. Let $\cat \coloneqq \vect$ be the category of vector spaces and let $\cat' = \cat'' \coloneqq \fdvect$ be the full subcategory of finite dimensional vector spaces. As every vector space is naturally isomorphic to the colimit of all its finite dimensional subspaces, taking the diagram of all the finite dimensional subspaces of a vector space gives a functor
	$$\delta:\vect\longrightarrow\ind{\fdvect}\ .$$
	It is well-known that $\vect$ is cocomplete and that the finite dimensional vector spaces are compact objects in $\vect$. Moreover, linear duality gives an anti-equivalence of $\fdvect$ with itself.  Therefore, by \cref{prop:CEquivToIndC} and \cref{lemma:antiEquivIndPro} we have that $\vect$ is anti-equivalent to $\pro{\fdvect}$.
\end{example}

\begin{example} \label{ex:GG}
	The following example is a result of \cite{getzler99}. Let $\cat\coloneqq\mathsf{Cog}$ be the category of coassociative counital coalgebras, let $\cat'\coloneqq \mathsf{fCog}$ be the full subcategory of finite dimensional coassociative counital coalgebras, and let $\cat''\coloneqq\mathsf{fAlg}$ be the category of finite dimensional unital associative algebras. One can prove that every coassociative counital coalgebra is the colimit of its finite dimensional subcoalgebras. By the same arguments as in the previous two examples, \cref{prop:CEquivToIndC} gives an equivalence of categories between $\mathsf{Cog}$ and $\ind{\mathsf{fCog}}$. Linear duality induces an anti-equivalence of categories between $\mathsf{fCog}$ and $\mathsf{fAlg}$. Therefore, by \cref{lemma:antiEquivIndPro} we have that $\mathsf{Cog}$ is anti-equivalent to $\pro{\mathsf{fAlg}}$.
\end{example}

\subsection{Model structures}

Model categories were first introduced in the book \cite{quillen67}, and have since then assumed a central role in mathematics. Another very clear and more modern introduction is given by the book \cite{hovey}.

\begin{definition}
	A \emph{closed model category} is a tuple $(\cat,W,C,F)$, where $\cat$ is a category and $W,C,F$ are three classes of morphisms in $\cat$, such that
	\begin{enumerate}
		\item the category $\cat$ is complete and cocomplete,
		\item the class $W$ contains all isomorphisms and has the \emph{2-out-of-3 property}, i.e. for any two composable morphisms  $f,g$ in $\cat$, if two of $f,g,gf$ are in $W$, then so is the third,
		\item all three classes $W,C,F$ are closed under retracts,
		\item if we have $i\in C$ and $p\in F$ such that one of them is in $W$ and a commutative square
		
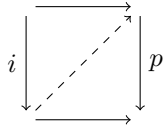
\begin{center}
			\begin{tikzpicture}
				\node (a) at (0,0) {};
				\node (b) at (1.5,0) {};
				\node (c) at (0,1.5) {};
				\node (d) at (1.5,1.5) {};
				
				\draw[->] (a) -- (b);
				\draw[->] (c) -- node[left]{$i$} (a);
				\draw[->] (c) -- (d);
				\draw[->] (d) -- node[right]{$p$} (b);
				
				\draw[->,dashed] (a) -- (d);
			\end{tikzpicture}
		\end{center}
		then there is a diagonal filler (the dashed arrow in the diagram above), and
		\item \label{CM5} every morphism $f$ in $\cat$ can be decomposed both as $f=pi$ with $p\in F$ and $i\in C\cap W$, and as $f=qj$ with $q\in F\cap W$ and $j\in C$.
	\end{enumerate}
	The elements of $W,C,F,C\cap W$ and $F\cap W$ are called \emph{weak equivalences}, \emph{cofibrations}, \emph{fibrations}, \emph{trivial} (or: \emph{acyclic}) \emph{cofibrations}, and \emph{trivial} (or: \emph{acyclic}) \emph{fibrations} respectively. We will usually omit the classes of arrows in the notation and simply speak of the model category $\cat$.
\end{definition}

\begin{remark}
	The definition given above is the original one of \cite{quillen67}. There exist various variations on the axioms throughout the literature. For example certain authors only ask that $\cat$ is closed under finite limits and colimits, and \cite{hovey} among others require that the factorization of axiom (\ref{CM5}) above be functorial. Usually, all results given for a specific set of axioms also hold for all other possible definitions of model category, with only minor modifications of the statements.
\end{remark}

Two sets of objects often play a central role in model categories.

\begin{definition}
	Let $\cat$ be a model category. An object $c\in\cat$ is said to be \emph{fibrant} if the unique morphism $\emptyset\to c$ from the initial object of $\cat$ to $c$ is a fibration. It is \emph{cofibrant} if the unique morphism $c\to*$ to the final object of $\cat$ is a cofibration.
\end{definition}

Under some assumptions, one can transfer a model structure from a category to another along an adjunction. We will not directly need this, but only the following easier result.

\begin{proposition} \label{prop:modelStrTransportedViaEquiv}
	Let $(\mathcal{C},W,C,F)$ be a closed model category, let $\mathcal{D}$ be another category, and let
	$$G:\mathcal{D}\longrightarrow\mathcal{C}$$
	be an equivalence of categories. Then the tuple $(\mathcal{D},G^{-1}(W),G^{-1}(C),G^{-1}(F))$ is a closed model category. Moreover, $G$ is a Quillen equivelence between the two model categories.
\end{proposition}

\begin{proof}
	One can always complete $G$ to an adjoint equivalence by \cref{thm:equivOfCats}. With this ulterior data, the proof is a straightforward verification of the axioms.
\end{proof}

\subsection{The homotopy category of a model category}

Given a category $\cat$ with a set of morphisms $W$ giving a ``good" notion of weak equivalence, one wants to consider the category $\ho{\cat}\coloneqq\cat[W^{-1}]$ obtained by formally inverting the arrows in $W$, see e.g. \cite[Def. 1.2.1]{hovey}. This in general results in an object that is not a category anymore because it might happen that the morphisms between two objects of $\cat[W^{-1}]$ do not form a set. However, if $\cat$ is a model category, then localization at $W$ gives a well defined category. We sketch the standard way of proving this here, following \cite[Ch. 1]{hovey}, as well as some results and definitions we will need throughout the text.

\begin{definition}
	Let $\cat$ be a model category, and let $x\in\cat$ be an object.
	\begin{enumerate}
		\item A \emph{cylinder object} for $x$ is a splitting of the fold map $\nabla:x\sqcup x\to x$ into
		$$x\sqcup x\longrightarrow I\times x\longrightarrow x\ ,$$
		with the first map a cofibration and the second map a weak equivalence.
		\item Dually, a \emph{path object} for $x$ is a splitting of the diagonal map $\Delta:x\to x\times x$ into
		$$x\longrightarrow x^I\longrightarrow x\times x\ ,$$
		with the first map a weak equivalence and the second map a fibration.
	\end{enumerate}
\end{definition}

Notice that, since we are in a model category, cylinder and path objects always exist.

\begin{lemma}
	Let $\cat$ be a model category. Any two choices of cylinder object for an object $x\in\cat$ are weakly equivalent. Dually, any two choices of path object are weakly equivalent.
\end{lemma}

\begin{proof}
	By the definition of model category, there exists a choice of cylinder object $x\times I$ as a trivial fibration. Suppose that $\tilde{x}$ is another cylinder for $x$. By lifting property of cofibrations with respect to trivial fibrations, there exists a diagonal filler for the following diagram.
	\begin{center}
		\begin{tikzpicture}
			\node (a) at (0,1.5) {$x\sqcup x$};
			\node (b) at (2,1.5) {$x\times I$};
			\node (c) at (0,0) {$\tilde{x}$};
			\node (d) at (2,0) {$x$};
			
			\draw[->] (a) -- (b);
			\draw[->] (b) -- (d);
			\draw[->] (a) -- (c);
			\draw[->] (c) -- (d);
			\draw[->,dashed] (c) -- (b);
		\end{tikzpicture}
	\end{center}
	It is a weak equivalence by $2$-out-of-$3$ property. The proof for path objects is dual.
\end{proof}

These objects can be used to define notions of homotopies between maps, in an analogous way to what is done in topology.

\begin{definition}
	Let $f,g:x\to y$ be two morphisms in the model category $\cat$.
	\begin{enumerate}
		\item The two morphisms are \emph{left homotopic} if there exists a map
		$$H:x\times I\longrightarrow y$$
		from a cylinder object for $x$ to $y$ such that the following diagram commutes.
		\begin{center}
			\begin{tikzpicture}
				\node (a) at (0,1.5) {$x$};
				\node (b) at (0,0) {$x\sqcup x$};
				\node (c) at (2,0) {$x\times I$};
				\node (d) at (4,0) {$y$};
				\node (e) at (0,-1.5) {$x$};
				
				\draw[->] (a)--(b);
				\draw[->] (e)--(b);
				\draw[->] (b)--(c);
				\draw[->] (c)-- node[above]{$H$} (d);
				\draw[->] (a) to[out=0,in=140] node[above]{$f$} (d);
				\draw[->] (e) to[out=0,in=-140] node[below]{$g$} (d);
			\end{tikzpicture}
		\end{center}
		We denote this by $f\sim_\ell g$.
		\item Dually, the morphisms are \emph{right homotopic} if there exists a moprhism $H:x\to y^I$ from $x$ to a path object for $y$ making the dual diagram commute. We denote this by $f\sim_r g$.
		\item If $f$ and $g$ are both left and right homotopic, then we simply say that $f$ and $g$ are \emph{homotopic} and write $f\sim g$.
	\end{enumerate}
\end{definition}

These notions of homotopy behave especially well when the source is cofibrant and the domain is fibrant.

\begin{proposition}
	Let $\cat$ be a model category, and let $x,y$ be objects of $\cat$.
	\begin{enumerate}
		\item If $x$ is cofibrant, then left homotopy is an equivalence relation on $\hom_\cat(x,y)$, being left homotopic implies being right homotopic, and right homotopies can be realized using any choice of path object.
		\item Dually, if $y$ is fibrant, then right homotopy is an equivalence relation on $\hom_\cat(x,y)$, being right homotopic implies being left homotopic, and left homotopies can be realized using any choice of cylinder object.
		\item In particular, if $x$ is cofibrant and $y$ is fibrant, then being left homotopic, right homotopic or homotopic are all equivalent notions and form an equivalence relation on $\hom_\cat(x,y)$.
	\end{enumerate}
\end{proposition}

\begin{proof}
	See e.g. \cite[Prop. 1.2.5]{hovey}.
\end{proof}

\begin{theorem}\label{thm:model categories}
	Let $\cat$ be a model category. Let $\cat_{cf}$ be the full subcategory of $\cat$ given by the objects that are both fibrant and cofibrant at the same time. The inclusion
	\[
	\cat_{cf}\longrightarrow\cat
	\]
	induces an equivalence of categories
	\[
	\cat_{cf}/{}_\sim\longrightarrow\ho{\cat}
	\]
	between the quotient of $\cat_{cf}$ by the homotopy relation and the homotopy category of $\cat$. Moreover, an arrow in $\cat_{cf}$ is a weak equivalence in $\cat$ if, and only if it is an isomorphism in $\ho{\cat}$.
\end{theorem}

\begin{proof}
	See e.g. \cite[Prop. 1.2.10]{hovey}.
\end{proof}

\section{The Goldman--Millson theorem and the Dolgushev--Rogers theorem} \label{sect:GMandDRThms}

The goal of this section is to introduce two classical results on the Maurer--Cartan locus of Lie algebras: the Goldman--Millson theorem and the Dolgushev--Rogers theorem. They play a very important role in deformation theory, see for example the original paper by W. M. Goldman and J. J. Millson \cite{goldman88}. These results are proven by working directly on the algebras coming into play, but both have a distinct homotopical flavor. It is one of the main objectives of the rest of the present paper to expose a new, completely homotopy theoretical approach to the proof of these theorems.

\subsection{Maurer--Cartan elements, gauge equivalences, and the Goldman--Millson theorem} \label{subsect:MCelGaugeEquiv}

A concept ubiquitous in mathematics today is that of the Maurer--Cartan elements of a Lie algebra.

\begin{definition}
	Let $\g$ be a Lie algebra. A \emph{Maurer--Cartan element} of $\g$ is a degree $1$ element $x\in\g^1$ satisfying the Maurer--Cartan equation
	\[
	dx + \frac{1}{2}[x,x] = 0\ .
	\]
	The set of Maurer--Cartan elements of $\g$ is denoted by $\MC(\g)$.
\end{definition}

Given $\lambda\in\g^0$, one can define a vector field $\g^1\to\g^1$ by
\[
x\longmapsto d\lambda + [\lambda,x]\ .
\]
It is tangent to $\MC(\g)$ in the sense that its flow --- whenever it exists --- preserves Maurer--Cartan elements. Let $x_0\in\MC(\g)$ and suppose that the flow
\[
\dot x(t) = d\lambda + [\lambda,x(t)],\qquad x(0) = x_0
\]
exists up to at least $t=1$. Then we say that $\lambda$ is a \emph{gauge} between $x_0$ and $x_1\coloneqq x(1)$, and that $x_0$ and $x_1$ are \emph{gauge equivalent}. We denote the equivalence relation generated by this notion of gauge equivalence by $\sim_{\mathrm{gauge}}$.

\medskip

One can guarantee that the flows as above exist for all times $t$ by requiring the Lie algebras one considers to be \emph{complete} with respect to some well-behaved filtration. As a first example, let $\g$ be a Lie algebra, then we define its \emph{canonical filtration} by
\[
F_n^\lie\g\coloneqq\left\{x\in\g\mid x\textrm{ can be obtained as a bracketing of at least }n\textrm{ elements of }\g\right\}\ .
\]
Notice that $F_n^\lie\g$ is at the same time a sub Lie algebra and a Lie ideal of $\g$.

\begin{definition}
	A Lie algebra $\g$ is \emph{complete with respect to its canonical filtration} if
	\[
	\g\cong\lim_{n}\g/F_n^\lie\g
	\]
	as Lie algebras.
\end{definition}

If $\g$ is complete with respect to its canonical filtration, then being gauge equivalent is an equivalence relation on Maurer--Cartan elements on the nose by the Baker--Campbell--Hausdorff formula.

\begin{definition}
	The \emph{moduli space of Maurer--Cartan elements} of a Lie algebra $\g$ is the quotient
	\[
	\MCbar(\g)\coloneqq\MC(\g)/\sim_\mathrm{gauge}
	\]
	of the space of Maurer--Cartan elements by the gauge relation.
\end{definition}

Notice that any morphism of Lie algebras induces a morphism between the moduli spaces of Maurer--Cartan elements. In fact,
\[
\MCbar:\left\{\g\in\lie\text{-}\mathsf{alg}\text{ complete with respect to its canonical filtration}\right\}\longrightarrow\sets
\]
is a functor from Lie algebras which are complete with respect to the canonical filtration to sets.

\medskip

One can generalize the canonical filtration defined above to a more general notion of filtration for Lie algebras.

\begin{definition}
	Let $\g$ be a Lie algebra. A \emph{(complete) filtration} for $\g$ is a sequence of subspaces
	\[
	\g\coloneqq F_0\g \supseteq F_1\g\supseteq F_2\g\supseteq F_3\g\supseteq\ldots
	\]
	such that
	\begin{enumerate}
		\item $\g\cong\lim_n\g/F_n\g$,
		\item $[F_i\g,F_j\g]\subseteq F_{i+j}\g$, and
		\item $d(F_n\g)\subseteq F_n\g$.
	\end{enumerate}
	A \emph{complete Lie algebra} is a couple $(\g,F_\bullet\g)$, where $\g$ is a Lie algebra and $F_\bullet\g$ is a filtration for $\g$.
\end{definition}

\begin{definition}
	Let $(\g,F_\bullet\g),(\h,F_\bullet\h)$ be complete Lie algebras. A morphism of Lie algebras $\phi:\g\to\h$ is \emph{filtered} if $\phi(F_n\g)\subseteq F_n\h$ for all $n$. It is a \emph{filtered quasi-isomorphism} if moreover it is a quasi-isomorphism and if the induced morphisms
	\[
	\g/F_n\g\longrightarrow\h/F_n\h
	\]
	are all quasi-isomorphisms. More generally, we say that a morphism of Lie algebras $\phi:\g\to\h$ is a filtered quasi-isomorphism if there exist filtrations on $\g$ and $\h$ making it into one.
\end{definition}

We will denote the category of complete Lie algebras and filtered morphisms by $\cdgl$.

\medskip

The Goldman--Millson theorem tells us how the moduli space of Maurer--Cartan elements behaves under filtered quasi-isomorphisms. It was first proven in \cite{goldman88}, then generalized in \cite{yekutieli12}. The version we expose here is the most modern and general one we are aware of, and it is a direct corollary of the Dolgushev--Rogers theorem \cite{dolgushev15}, which we will present shortly.

\begin{theorem}[Goldman--Millson] \label{thm:GM}
	Let $\g,\h$ be complete Lie algebras, and let $\phi:\g\to\h$ be a filtered quasi-isomorphism. Then
	\[
	\MCbar(\phi):\MCbar(\g)\longrightarrow\MCbar(\h)
	\]
	is a bijection.
\end{theorem}

The proof of this theorem is an induction on nilpotence degrees using obstruction theory. One starts by proving that the statement is always true for abelian Lie algebras, which implies that it is also true for the induced maps
\[
\g/F_2\g\longrightarrow\h/F_2\h
\]
and
\[
F_n\g/F_{n+1}\g\longrightarrow F_n\h/F_{n+1}\h\ ,
\]
for any $n\ge2$. Then one can proceed by induction, where the induction step is done by using some obstruction theory on the morphism of short exact sequences
\[
\begin{CD}
	0 @>>> \frac{F_n\g}{F_{n+1}\g} @>>> \frac{\g}{F_{n+1}\g} @>>> \frac{\g}{F_n\g} @>>> 0\\
	@. @VVV @VVV @VVV\\
	0 @>>> \frac{F_n\h}{F_{n+1}\h} @>>> \frac{\h}{F_{n+1}\h} @>>> \frac{\h}{F_n\h} @>>> 0
\end{CD}
\]
induced by $\phi$.

\begin{remark}
	A more homotopical proof of the Goldman--Millson theorem has recently appeared in \cite{buijs16}. We will look at it in some detail in \cref{subsect:relations with BFMT}.
\end{remark}

\subsection{The deformation $\infty$-groupoid and the Dolgushev--Rogers theorem} \label{subsect:SullivanAlg}

The Goldman--Millson theorem is in fact only an avatar of a ``higher" result. In \cite{hinich97}, Hinich introduced a simplicial set $\MC_\bullet(\g)$ associated to a complete Lie algebra encoding not only Maurer--Cartan elements and gauges between them, but also all higher relations between those. This object was subsequently generalized and studied in depth by Getzler in \cite{getzler09}.

\medskip

Denote by $\Omega_\bullet$ the Sullivan algebra. It is the simplicial commutative algebra of polynomial differential forms on the $n$-simplex. Namely, we have
\[
\Omega_n\coloneqq\frac{\k[t_0,\ldots,t_n,dt_0,\ldots,dt_n]}{\left(\sum_{i=0}^nt_i = 1,\ \sum_{i=0}^ndt_i = 0\right)}
\]
endowed with the unique square zero differential extending $d(t_i) = dt_i$, and with the obvious simplicial maps induced by pullback along the inclusion of lower dimensional simplices and the collapsing of simplices along subsimplices. It was first introduced by Sullivan in \cite{sullivan77} in the context of rational homotopy theory.

\medskip

Let $\g$ be a nilpotent Lie algebra, i.e. a Lie algebra so that there exists an integer $N>0$ such that any bracketing of $N$ elements of $\g$ gives zero. The smallest such $N$ is called the \emph{nilpotency degree} of $\g$. The tensor product $\g\otimes\Omega_n$ is again a Lie algebra by
\[
[x\otimes\alpha,y\otimes\beta] = [x,y]\otimes \alpha\beta\ .
\]
Therefore, $\g\otimes\Omega_\bullet$ is a simplicial Lie algebra by tensoring the simplicial maps of $\Omega_\bullet$ by the identity on $\g$. More generally, if $(\g,F_\bullet\g)$ is a complete Lie algebra, we will denote
\[
\g\otimes\Omega_\bullet\coloneqq\lim_n\left(\left(\g/F_n\g\right)\otimes\Omega_\bullet\right)\ .
\]

\begin{definition} \label{def:DHGgpd}
	Let $\g$ be a complete Lie algebra. The \emph{deformation $\infty$-groupoid} of $\g$ --- also called the \emph{Deligne--Hinich--Getzler $\infty$-groupoid} --- is the simplicial set
	\[
	\MC_\bullet(\g)\coloneqq\MC(\g\otimes\Omega_\bullet)\ .
	\]
\end{definition}

Notice that $\MC_0(\g) = \MC(\g)$, while $\MC_1(\g)$ gives a notion of gauge between two Maurer--Cartan elements which --- although slightly different --- is equivalent to the one defined in \cref{subsect:MCelGaugeEquiv}, see for example \cite[Prop. 4.3]{rn17cosimplicial}. In particular, we have
\[
\pi_0\MC_\bullet(\g) \cong \MCbar(\g)\ .
\]
The assignment of the Deligne--Hinich--Getzler $\infty$-groupoid to a Lie algebra gives a functor
\[
\MC_\bullet:\cdgl\longrightarrow\ssets\ .
\]
The Deligne--Hinich--Getzler $\infty$-groupoid of a Lie algebra has nice homotopical properties: it is always fibrant in the category of simplicial sets, as implied by the following theorem.

\begin{theorem}[{\cite[Prop. 2.2.3]{hinich97}}, {\cite[Prop. 4.7]{getzler09}}, and {\cite[Thm. 2]{rogers16}}]\label{thm:Getzler}
	The functor $\MC_\bullet$ has image in the full subcategory of Kan complexes.
\end{theorem}

The analogue to the Goldman--Millson theorem in this context is the Dolgushev--Rogers theorem \cite[Thm. 2.2]{dolgushev15}.

\begin{theorem}[Dolgushev--Rogers] \label{thm:DR}
	Let $\g,\h$ be complete Lie algebras, and let $\phi:\g\to\h$ be a filtered quasi-isomorphism. Then
	\[
	\MC_\bullet(\phi):\MC_\bullet(\g)\longrightarrow\MC_\bullet(\h)
	\]
	is a homotopy equivalence of simplicial sets.
\end{theorem}

Again, the proof of this theorem is done by induction on the nilpotence degree of the Lie algebras, with the help some obstruction theory and some simplicial homotopy theory (see the original paper \cite{dolgushev15} for a clear exposition of all the details). What is striking to us here is the fact that this theorem has a distinct homotopical flavor --- its statement is of the form: ``morphisms in such and such class are sent to weak equivalences of simplicials sets by the functor $\MC_\bullet$" --- but its proof is not really in this context and relies on different methods.

\begin{remark}
	In fact, the material presented in this last subsection is true in a more general context: $\L_\infty$-algebras and $\infty$-morphisms between them, see \cite{getzler09} and \cite{dolgushev15}. A generalization of the methods of the present article to treat the more general case could be an interesting research direction.
\end{remark}

\section{The Vallette model structure on conilpotent coalgebras over a cooperad} \label{sect:ValletteModelStr}

In this section, we present the Vallette model structure on conilpotent coalgebras over the Koszul dual cooperad of a Koszul operad $\P$, see \cite{vallette14}. It generalizes the model structure on conilpotent coassociative coalgebras defined by Positselski \cite[Sect. 9.3]{positselski11}. We further develop the material found in Vallette's article and completely characterize the class weak equivalences in terms of filtered quasi-isomorphisms. Finally, we specialize to $\P=\susp\otimes\com$ and study the model structure induced on conilpotent Lie coalgebras.

\subsection{The Vallette model structure}

In order to make certain operadic results work, one needs an additional grading on operads and cooperads.

\begin{definition}
	Let $\P$ be an operad. A \emph{connected weight grading} on $\P$ is a decomposition
	\[
	\P = \k\id\oplus\P^{(1)}\oplus\P^{(2)}\oplus\cdots
	\]
	into subspaces $\P^{(\omega)}$ of \emph{weight $\omega$} such that the total weight is preserved both by the differential and by the composition map of $\P$. Dually, a \emph{connected weight grading} on a cooperad $\C$ is a similar decomposition such that the differential and the decomposition map of $\C$ preserve the total weight. Morphisms of (co)operads and twisting morphisms between connected weight graded (co)operads must preserve the weight.
\end{definition}

Let $(E,R)$ be an operadic quadratic data, see \cite[Sect. 7.1]{vallette12} for definitions and basic facts. Suppose that we have a decomposition of the generating $\S$-module
\[
E = E^{(1)}\oplus E^{(2)}\oplus E^{(3)}\oplus\cdots
\]
into subspaces $E^{(\omega)}$ of weight $\omega$ such that the relations in $R$ are of homogeneous weight. Then the weight grading on $E$ induces a connected weight grading on both the quadratic operad $\P\coloneqq\P(E,R)$ and the quadratic cooperad $\P^{\antishriek}\coloneqq\C(E,R)$, and the twisting morphism
\[
\kappa:\P^{\antishriek}\longrightarrow\P
\]
respects this weight grading.

\begin{remark}
	If $E(0)=E(1)=0$, then a good weight grading is obtained by imposing $E^{(\omega)}=E(\omega+1)$, and the induced connected weight gradings on the (co)operad are given by
	\[
	\P^{(\omega)} = \P(\omega+1),\qquad\text{and}\qquad(\P^{\antishriek})^{(\omega)}=\P^{\antishriek}(\omega+1)\ .
	\]
	More generally, if an operad $\P$ is \emph{connected}, i.e. if we have $\P(0)=0$ and $\P(1)=\k\id$, then
	\[
	\P^{(\omega)} = \P(\omega+1)
	\]
	is a canonical weight grading on $\P$, and similarly for cooperads.
\end{remark}

For the rest of this section, fix a Koszul operad $\P$ with a connected weight grading obtained as above from a weight grading on the quadratic data, and denote
$$\kappa:\P^{\antishriek}\longrightarrow\P$$
the canonical twisting morphism given by Koszul duality. The category of $\P$-algebras is given the Hinich model structure \cite{hinich97haha}, where the weak equivalences are the quasi-isomorphisms and the fibrations are the surjections.

\begin{remark}
	The theory presented below has been vastly generalized by Le Grignou in \cite{legrignou16}: one can consider a twisting morphism
	\[
	\alpha:\C\longrightarrow\P
	\]
	from a curved conilpotent cooperad to an operad and obtain a model structure on the category of conilpotent $\C$-coalgebras by pulling back weak equivalences and cofibrations via the cobar construction $\Omega_\alpha$ induced by $\alpha$.
\end{remark}

The twisting morphism $\kappa$ induces a bar-cobar adjunction
$$\adjunction{\Omega\kappa}{\mathsf{conil.\ }\P^{\antishriek}\text{-}\mathsf{cog.}}{\P\text{-}\mathsf{alg}}{\bar_\kappa}$$
between conilpotent $\P^{\antishriek}$-coalgebras and $\P$-algebras.

\begin{theorem}[{\cite[Th. 2.1]{vallette14}}] \label{thm:ValletteModelStr}
	The following three classes of morphisms define a model structure on the category $\mathsf{conil.\ }\P^{\antishriek}\text{-}\mathsf{cog.}$ of conilpotent $\P^{\antishriek}$-coalgebras.
	\begin{itemize}
		\item The class $W$ of weak equivalences is given by all morphism in $\mathsf{conil.\ }\P^{\antishriek}\text{-}\mathsf{cog.}$ that are sent to quasi-isomorphisms by $\Omega_\kappa$.
		\item The class $C$ of cofibrations is given by the injective morphisms.
		\item The class $F$ of fibrations is given by all the morphisms that have the right lifting property with respect to $C\cap W$.
	\end{itemize}
	With this model structure, all conilpotent $\P^{\antishriek}$-coalgebras are cofibrant, and the fibrant objects are the quasi-free $\P^{\antishriek}$-coalgebras. Moreover, the bar-cobar adjunction is a Quillen equivalence between $\mathsf{conil.\ }\P^{\antishriek}\text{-}\mathsf{cog.}$ and $\P\text{-}\mathsf{alg}$.
\end{theorem}

Vallette further proves that all weak equivalences of conilpotent $\P^{\antishriek}$-coalgebras are quasi-isomorphisms \cite[Prop. 2.5]{vallette14} and that the quasi-isomorphisms of $\P^{\antishriek}$-coalgebras respecting the coradical filtration are weak equivalences \cite[Prop. 2.3]{vallette14}. We will use these results to give a complete characterization of the class of weak equivalences in this model structure. The proof is analogous of the one given by Positselski \cite[Sect. 9.3]{positselski11} in the case of coassociative coalgebras, and all the ingredients are already present in \cite{vallette14}.

\subsection{Characterization of the weak equivalences}

\begin{definition} \label{def:filtrationOfCCog}
	Let $C$ be a conilpotent $\P^{\antishriek}$-coalgebra. A \emph{filtration} on $C$ is a sequence of sub-chain complexes
	$$0=F_{-1}C\subseteq F_0C\subseteq F_1C\subseteq\ldots\subseteq C$$
	satisfying \cite[Prop. 2.2]{vallette14}. Namely, we require that
	\begin{enumerate}
		\item the filtration is exhaustive: $\colim{n} F_nC\cong C$,
		\item it respects the coproduct, that is
		\[
		\overline{\Delta}_C(F_nC)\subseteq\bigoplus_{\substack{1\le w\le n,\ k\ge1\\n_1+\cdot+n_k = n-w}}\left(\P^{\antishriek}(k)^{(w)}\otimes F_{n_1}C\otimes\cdots\otimes F_{n_k}C\right)^{\S_k},
		\]
		where $\overline{\Delta}_C(c)\coloneqq\Delta_C(c)-c$, and
		\item it is preserved by the differential: $d_C(F_nC)\subseteq F_nC$.
	\end{enumerate}
\end{definition}

\begin{example}
	A filtration that exists for any conilpotent $\P^{\antishriek}$-coalgebra $C$ is the \emph{coradical filtration}. It is defined by
	$$F_n^{\mathrm{corad.}}C\coloneqq\left\{c\in C\mid\Delta_C(c)\in\bigoplus_{w=0}^n{\P^{\antishriek}}^{(w)}(C)\right\}\ .$$
\end{example}

\begin{definition}
	Let $C_1,C_2$ be two conilpotent $\P^{\antishriek}$-coalgebras, and let $\phi:C_1\to C_2$ be a morphism of $\P^{\antishriek}$-coalgebras.
	\begin{enumerate}
		\item Let $F_\bullet C_1$ and $F_\bullet C_2$ be filtrations. The morphism $\phi$ is \emph{filtered} (with respect to the given filtrations) if for each $n\ge0$ we have $\phi(F_nC_1)\subseteq F_nC_2$.
		\item The morphism $\phi$ is a \emph{filtered quasi-isomorphism} if it is a quasi-isomorphism and if there are filtrations $F_\bullet C_1$ and $F_\bullet C_2$ such that $\phi$ is filtered and such that for each $n\ge0$, the morphism $\phi$ induces a quasi-isomorphism
		\[
		F_nC_1/F_{n-1}C_1\longrightarrow F_nC_2/F_{n-1}C_2\ .
		\]
	\end{enumerate}
\end{definition}

\begin{remark}
	Notice that the coradical filtration is final, in the sense that if we put the coradical filtration on $C_1$ and any filtration on $C_2$ then any morphism of $\P^{\antishriek}$-coalgebras will be filtered.
\end{remark}

\begin{theorem} \label{thm:charWE}
	The class $W$ of weak equivalence is the smallest class of arrows of $\mathsf{conil.\ }\P^{\antishriek}\text{-}\mathsf{cog.}$ containing all filtered quasi-isomorphisms and which is closed under the $2$-out-of-$3$ property.
\end{theorem}

As already mentioned, the proof of this theorem is similar to what can be found in \cite[Sect. 9.3]{positselski11}. Before going on, we need a couple of preliminary results, all of which come from \cite{vallette14}.

\begin{lemma} \label{lemma:OmegaOfFQIisQI}
	Let $f$ be a filtered quasi-isomorphism of conilpotent $\P^{\antishriek}$-coalgebras. Then the morphism of $\P$-algebras $\Omega_\kappa f$ is a quasi-isomorphism.
\end{lemma}

\begin{proof}
	This follows from the proof of \cite[Prop. 2.3]{vallette14}.
\end{proof}

\begin{lemma} \label{lemma:BarOfQIisFQI}
	Let $f$ be a quasi-isomorphism of $\P$-algebras. Then $\bar_\kappa f$ is a filtered quasi-isomorphism.
\end{lemma}

\begin{proof}
	This is \cite[Prop. 2.4]{vallette14}.
\end{proof}

\begin{lemma} \label{lemma:counitIsFQI}
	Let $C$ be a conilpotent $\P^{\antishriek}$-coalgebra. Then the unit map
	$$\eta_C:\bar_\kappa\Omega_\kappa C\longrightarrow C$$
	is a filtered quasi-isomorphism.
\end{lemma}

\begin{proof}
	This is a consequence of the proof of \cite[Thm. 2.6]{vallette14}.
\end{proof}

\begin{proof}[Proof of \cref{thm:charWE}]
	We denote by $\fqi$ the smallest class of arrows in $\mathsf{conil.\ }\P^{\antishriek}\text{-}\mathsf{cog.}$ which contains all filtered quasi-isomorphisms and which is closed under the $2$-out-of-$3$ property.
	
	\medskip
	
	\cref{lemma:OmegaOfFQIisQI} implies that that $\fqi\subseteq W$. To prove the other inclusion, let
	$$f:C\longrightarrow D$$
	be a morphism of conilpotent $\P^{\antishriek}$-coalgebras such that $\Omega_\kappa f$ is a quasi-isomorphism, that is to say $f\in W$. We consider the diagram
	\begin{center}
		\begin{tikzpicture}
			\node (a) at (0,0){$C$};
			\node (b) at (3,0){$D$};
			\node (c) at (0,2){$\bar_\kappa\Omega_\kappa C$};
			\node (d) at (3,2){$\bar_\kappa\Omega_\kappa D$};
			
			\draw[->] (a)-- node[above]{$f$} (b);
			\draw[->] (c)-- node[above]{$\bar_\kappa\Omega_\kappa f$} (d);
			\draw[->] (c)-- node[left]{$\eta_C$} (a);
			\draw[->] (d)-- node[right]{$\eta_D$} (b);
		\end{tikzpicture}
	\end{center}
	where by \cref{lemma:BarOfQIisFQI} the arrow $\bar_\kappa\Omega_\kappa f$ is a filtered quasi-isomorphism, and by \cref{lemma:counitIsFQI} both vertical arrows are also filtered quasi-isomorphisms. A double application of the $2$-out-of-$3$ property proves that $f\in\fqi$, concluding the proof.
\end{proof}

One could consider a twisting morphism different from the one given by Koszul duality and try induce a model structure on conilpotent $\C$-coalgebras using the associated bar-cobar adjunction. This can indeed be done (as long as the characteristic of the base field is $0$), as is shown in \cite[Thm. 10]{legrignou16}. Then one can try to compare two model structures on conilpotent $\C$-coalgebras obtained this way. This is done in \cite[Prop. 32]{legrignou16}. We will need the following special case.

\begin{proposition}[{\cite[Prop. 32]{legrignou16}}] \label{prop:compareModelStr}
	Let $\C$ be a cooperad, let $\P,\P'$ be operads and suppose we have a diagram
	\begin{center}
		\begin{tikzpicture}
			\node (a) at (0,0){$\C$};
			\node (b) at (2,1){$\P$};
			\node (c) at (2,-1){$\P'$};
			
			\draw[->] (a) -- node[above]{$\alpha$} (b);
			\draw[->] (b) -- node[right]{$f$} (c);
			\draw[->] (a) -- node[below]{$f\alpha$} (c);
		\end{tikzpicture}
	\end{center}
	where $\alpha:\C\to\P$ is a twisting morphism and $f$ is a quasi-isomorphism. Then the model structures on conilpotent $\C$-coalgebras induced by $\alpha$ and $f\alpha$ are equal.
\end{proposition}

\subsection{Case of interest: conilpotent Lie coalgebras} \label{subsect:coLieCog}

We are interested in the study of certain dg Lie algebras, namely the ones obtained dualizing conilpotent Lie coalgebras.

\begin{definition}
	A \emph{conilpotent Lie coalgebra} is a conilpotent coalgebra over the cooperad $\colie\coloneqq\lie^\vee$, the linear dual of the operad $\lie$ encoding Lie algebras.
\end{definition}

Notice that the cooperad $\colie$ is Koszul and that its Koszul dual operad is the operad $\susp\otimes\com$. Both are connected, and we endow them with the canonical weight grading. By what exposed above, we obtain a model structure on the category of conilpotent Lie coalgebras by pulling back the Hinich model structure on (suspended) commutative algebras along the Koszul twisting morphism
$$\kappa:\colie\longrightarrow\susp\otimes\com\ .$$
We fix the canonical weight grading on $\colie$. It is given by
\[
\colie^{(w)}\coloneqq\colie(w+1)\ .
\]
Therefore, we can change our conventions a bit and we define a \emph{filtration} on a conilpotent Lie coalgebra $C$ to be a sequence of sub-chain complexes
\[
0\subseteq F_1C\subseteq F_2C\subseteq F_3C\subseteq\cdots\subseteq C
\]
such that
\begin{enumerate}
	\item the filtration is exhaustive: $\colim{n}F_nC\cong C$,
	\item it respects the coproduct, that is
	\[
	\overline{\Delta}_C(F_nC)\subseteq\bigoplus_{\substack{k\ge1\\n_1+\cdots+n_k = n}}\big(\colie(k)\otimes F_{n_1}C\otimes\cdots\otimes F_{n_k}C\big)^{\S_k},
	\]
	and
	\item it is preserved by the differential: $d_C(F_nC)\subseteq F_nC$.
\end{enumerate}
Notice that this corresponds to the notion of filtration of \cref{def:filtrationOfCCog} by a shift of index by $1$.

\medskip

Explicitly, the model structure on the category $\mathsf{conil.\ coLie}$ of conilpotent Lie coalgebras has
\begin{itemize}
	\item the closure of the class of filtered quasi-isomorphisms under the $2$-out-of-$3$ property as weak equivalences,
	\item the monomorphisms as cofibrations, and
	\item the class of morphisms with the right lifting property with respect to trivial cofibrations as fibrations.
\end{itemize}
Every conilpotent Lie coalgebra is cofibrant, and the fibrant objects are the quasi-free conilpotent Lie coalgebras.

\medskip

Dualizing a filtered conilpotent Lie coalgebra we obtain a complete Lie algebra.

\begin{lemma}
	Let $(C,F_\bullet C)$ be a filtered conilpotent Lie coalgebra. Then the filtration on $\g\coloneqq C^\vee$ defined by
	\[
	F_n\g\coloneqq(F_nC)^\perp
	\]
	makes $\g$ into a complete Lie algebra. Moreover, the dual of a filtered quasi-isomorphism of conilpotent Lie coalgebras is a filtered quasi-isomorphism of Lie algebras.
\end{lemma}

\begin{remark}\label{rem:not all filtered qi are duals}
	Not all filtered quasi-isomorphisms of Lie algebras can be obtained by linear dualization. For example, consider abelian (co)Lie (co)algebras, i.e. cochain complexes. For simplicity, take $\k=\mathbb{Q}$. Let
	\[
	V\coloneqq\bigoplus_{n\in\mathbb{Z}}\mathbb{Q}
	\]
	be concentrated in degree $0$, seen as an abelian Lie coalgebra. A quasi-isomorphism from $V$ to itself is just an isomorphism, and we have
	\[
	|\operatorname{Aut}(V)|\le|\operatorname{End}(V)|=(\dim V)^{|V|} = |\mathbb{N}|^{|\mathbb{N}|}.
	\]
	At the same time, we have $\dim(V^\vee) = |\mathbb{R}|$. Notice that if we fix a basis of $V^\vee$, then bijections from the basis to itself are automorphisms. Therefore, we have
	\[
	|\operatorname{Aut}(V^\vee)|\ge|\operatorname{Sym}(\mathbb{R})|=|\mathbb{R}|^{|\mathbb{R}|}>|\mathbb{N}|^{|\mathbb{N}|}.
	\]
	It follows that there must exist automorphisms of $V^\vee$ that are not given as the dual of an automorphism of $V$.
\end{remark}

There is another twisting morphism we can consider to construct a model structure on conilpotent Lie coalgebras, namely the canonical twisting morphism
\begin{equation} \label{eq:iota}
	\iota:\colie\cong\susp^c\otimes\com^{\antishriek}\longrightarrow\Omega(\susp^c\otimes\com^{\antishriek})\cong\susp\otimes\C_\infty\ .
\end{equation}

\begin{lemma} \label{lemma:modelStrKappaIota}
	The model structure obtained on conilpotent Lie coalgebras by pulling back the Hinich model structure on suspended $\C_\infty$-algebras along $\Omega_\iota$ is equal to the one obtained by pulling back along $\Omega_\kappa$.
\end{lemma}

\begin{proof}
	This is an immediate corollary of \cref{prop:compareModelStr}.
\end{proof}

This fact will be of fundamental importance later on.

\section{A model structure for the Goldman--Millson theorem} \label{sect:GMThm}

In this section, we show that a version of the Goldman--Millson theorem \cite{goldman88} can be obtained by a purely homotopical argument in the model category of Lie coalgebras introduced above.

\subsection{Dualization of conilpotent Lie coalgebras}

We start by giving the correct framework in which to dualize conilpotent Lie coalgebras.

\medskip

In the article \cite{legrignou16}, the following result is proven, generalizing what can be found in \cite[Sect. 1]{getzler99} for coassociative counital coalgebras. See also \cref{ex:GG}.

\begin{proposition}[{\cite[Lemma 4 and 5 and Prop. 12]{legrignou16}}]
	Let $\C$ be a cooperad. The category $\mathsf{conil.\ }\C\text{-}\mathsf{cog}$ of conilpotent $\C$-coalgebras is cocomplete, every conilpotent $\C$-coalgebra is the colimit of the diagram of all its finite-dimensional sub-coalgebras with the relative inclusions, and all finite-dimensional conilpotent $\C$-coalgebras are compact objects in $\mathsf{conil.\ }\C\text{-}\mathsf{cog}$. In particular, the category $\mathsf{conil.\ }\C\text{-}\mathsf{cog}$ is presentable.
\end{proposition}

Denote by $\mathsf{fd\ conil.\ coLie}$ the full subcategory of $\mathsf{conil.\ coLie}$ of finite dimensional conilpotent Lie coalgebras. As a direct consequence of these facts and of \cref{prop:CEquivToIndC} in the case where $\cat = \mathsf{conil.\ coLie}$, we have the following.

\begin{corollary} \label{cor:equivCats}
	The categories $\mathsf{conil.\ coLie}$ and $\ind{\mathsf{fd\ conil.\ coLie}}$ are equivalent. The equivalence is given by the functors
	$$\colim{}:\ind{\mathsf{fd\ conil.\ coLie}}\longrightarrow\mathsf{conil.\ coLie}$$
	given by taking a functorial choice of colimit of diagrams and
	$$\delta:\mathsf{conil.\ coLie}\longrightarrow\ind{\mathsf{fd\ conil.\ coLie}}$$
	given by associating to a conilpotent Lie coalgebra the diagram of all its finite-dimensional sub-coalgebras with the relative inclusions.
\end{corollary}

It is always true that the linear dual of a conilpotent Lie coalgebra is a Lie algebra, while the converse holds in the finite dimensional case. Denote by $\mathsf{fd\ nil.\ Lie}$ the category of finite dimensional nilpotent Lie algebras. As a direct consequence of this fact and of \cref{lemma:antiEquivIndPro}, we have the following.

\begin{lemma}
	Linear duality between finite dimensional Lie coalgebras and finite dimensional Lie algebras induces an anti-equivalence of categories between $\ind{\mathsf{fd\ conil.\ coLie}}$ and $\pro{\mathsf{fd\ nil.\ Lie}}$.
\end{lemma}

In summary, we have the following commutative diagram, where $\mathsf{pronil.\ Lie}$ is the subcategory of Lie algebras which is the image of linear dualization from conilpotent Lie coalgebras, called the category of \emph{pronilpotent Lie algebras}.
\begin{center}
	\begin{tikzpicture}
		\node (a) at (0,0) {$\ind{\mathsf{fd\ conil.\ coLie}}$};
		\node (b) at (5,0) {$\pro{\mathsf{fd\ nil.\ Lie}}$};
		\node (c) at (0,-2) {$\mathsf{conil.\ coLie}$};
		\node (d) at (5,-2) {$\mathsf{pronil.\ Lie}$};
		
		\draw[->] (c) -- (d);
		\draw[->] (a) -- node[right]{$\colim{}$} (c);
		\draw[->] (b) -- node[right]{$\lim$} (d);
		\draw[->] (b) -- (a);
	\end{tikzpicture}
\end{center}
where the horizontal arrows are given by linear dualization and the upper horizontal and left vertical arrows are an anti-equivalence of categories and an equivalence of categories respectively. Notice that all pronilpotent Lie algebras are complete: one can always write them as the dual of a conilpotent Lie coalgebras an take the orthogonal of the coradical filtration.

\begin{remark}
	What stated above works in greater generality for conilpotent coalgebras over a cooperad. In a sense, this is the correct framework in which to dualize a conilpotent coalgebra, obtaining not an algebra, but a diagram of finite dimensional algebras.
\end{remark}

We endow the category $\pro{\mathsf{fd\ nil.\ Lie}}$ with the model structure obtained by transporting the Vallette model structure described in \cref{thm:ValletteModelStr} along the two (anti-)equivalences of categories using \cref{prop:modelStrTransportedViaEquiv}. The two model structures are Quillen equivalent. Notice that fibrations and cofibrations change roles as we pass through an anti-equivalence.

\begin{remark}
	One cna check that the model structure thus obtained is exactly the one of \cite[Def. 9.9]{lazarev13}. The main difference between our approach and theirs is the fact that we work directly with Lie coalgebras, while they only ever work with pronilpotent Lie algebras. We hope that the point of view we used in the present paper will appear more natural to some readers, as it avoids dualizing every coalgebra in order to only work with algebras.
\end{remark}

\subsection{The cosimplicial Lie algebra $\mc_\bullet$} \label{subsect:mcbullet}

Before going on, we need to introduce a particular cosimplicial Lie algebra playing a very central role in deformation theory and rational homotopy theory. It was discovered and studied in \cite{bandiera}, \cite{buijs15}, \cite{rn17cosimplicial}, and \cite{buijs17}.

\medskip

Recall the Sullivan algebra $\Omega_\bullet$ described in \cref{subsect:SullivanAlg}. The cocellular complex $C_n$ of the $n$-simplex can be given by a natural sub-cochain complex of $\Omega_n$. Namely, it is given by
\[
C_n\coloneqq\mathrm{span}_\k\big\{\omega_I\mid I\subseteq\{0,1,\ldots,n\},\ I\neq\emptyset\big\}
\]
with
\[
\omega_I\coloneqq k!\sum_{j=0}^k(-1)^jt_jdt_{i_0}\cdots\widehat{dt_{i_j}}\cdots dt_{i_k}
\]
for $I=\{i_0<i_1<\cdots<i_k\}$. It forms a simplicial cochain complex by restriction of the simplicial maps of $\Omega_\bullet$. For his proof of the de Rham theorem for PL differential forms, Dupont \cite{dupont76} introduced a simplicial contraction
\begin{center}
	\begin{tikzpicture}
		\node (a) at (0,0){$\Omega_\bullet$};
		\node (b) at (2,0){$C_\bullet$};
		
		\draw[->] (a)++(.3,.1)--node[above]{\mbox{\tiny{$p_\bullet$}}}+(1.4,0);
		\draw[<-,yshift=-1mm] (a)++(.3,-.1)--node[below]{\mbox{\tiny{$i_\bullet$}}}+(1.4,0);
		\draw[->] (a) to [out=-150,in=150,looseness=4] node[left]{\mbox{\tiny{$h_\bullet$}}} (a);
	\end{tikzpicture}
\end{center}
from $\Omega_\bullet$ to $C_\bullet$. Thanks to this, the Homotopy Transfer Theorem (see e.g. \cite[Sect. 10.3]{vallette12} for a standard reference) gives us a good way to endow $C_\bullet$ with a simplicial $\C_\infty$-algebra structure. Since $\C_\infty$ is finite dimensional in every arity, its dual is a cooperad, which is in fact isomorphic to $\susp^{-1}\otimes\bar\lie$. Since $C_n$ is finite dimensional for every $n$, we can dualize it to obtain a (not conilpotent) cosimplicial $\C_\infty^\vee$-coalgebra $C_\bullet^\vee$. Its suspension $sC_\bullet^\vee$ is then a $\bar\lie$-coalgebra, so that we can take its complete cobar construction relative to the canonical twisting morphism
\[
\pi:\bar\lie\longrightarrow\lie
\]
to obtain a cosimplicial Lie algebra
\[
\mc_\bullet\coloneqq\widehat{\Omega}_\pi(sC_\bullet^\vee)\ .
\]
The underlying graded Lie algebra of $\mc_\bullet$ is simply the free complete Lie algebra
\[
\widehat{\lie}(sC_\bullet^\vee)\coloneqq\prod_{n\ge1}\lie(n)\otimes_{\S_n}(sC_\bullet^\vee)^{\otimes n}.
\]
Notice that
\[
\mc_\bullet = \bar_\iota(s^{-1}C_\bullet)^\vee\ ,
\]
where $\iota$ is the canonical twisting morphism of (\ref{eq:iota}). Details on this construction can be found in \cite{rn17tensor} and \cite{rn17cosimplicial}. From the latter article we also recall the following two results. The first tells us that the cosimplicial Lie algebra $\mc_\bullet$ represents the Deligne--Hinich--Getzler $\infty$-groupoid of \cref{def:DHGgpd}. The second one explicitly describes $\mc_0$ and $\mc_1$.

\begin{theorem}[{\cite[Cor.5.3]{rn17cosimplicial}}] \label{thm:equivDHGgpd}
	Let $\g$ be a complete Lie algebra. There is a homotopy equivalence
	\[
	\MC_\bullet(\g)\simeq\hom_{\dgl}(\mc_\bullet,\g)
	\]
	which is natural in $\g$.
\end{theorem}

\begin{remark}
	In fact, by a result of Bandiera \cite[Prop. 2.5]{bandiera17}, the Kan complex $\hom_{\dgl}(\mc_\bullet,\g)$ is isomorphic to the Kan complex $\gamma_\bullet(\g)$ defined by Getzler in \cite{getzler09}.
\end{remark}

\begin{proposition}[{\cite[Prop. 5.4]{rn17cosimplicial}}] \label{prop:mc01}
	For the first two levels of the cosimplicial Lie algebra $\mc_\bullet$ we have that
	\begin{enumerate}
		\item $\mc_0$ is isomorphic to the Lie algebra generated by a single Maurer--Cartan element, and
		\item $\mc_1$ is isomorphic to the Lawrence--Sullivan algebra, that is the Lie algebra generated by two Maurer--Cartan elements in degree $1$ and an element in degree $0$ giving a gauge between them.
	\end{enumerate}
\end{proposition}

\subsection{The Goldman--Millson theorem}

The idea is now to use the model structure given above to obtain statements about pronilpotent Lie algebras. Namely, by \cref{prop:mc01} we have that the set of Maurer--Cartan elements of a Lie algebra $\g$ is in natural bijection with the set of morphisms
$$\mc_0\longrightarrow\g$$
of dg Lie algebras, where $\mc_0$ is the free Lie algebra on a Maurer--Cartan element. Similarly, gauge equivalences between Maurer--Cartan elements are coded by morphisms
$$\mc_1\longrightarrow\g\ .$$
Intuitively, $\mc_1$ behaves like a cylinder object for $\mc_0$, so that the gauge relation looks very similar to a left homotopy between Maurer--Cartan elements seen as maps. However, we do not have a good model structure on complete Lie algebras, and that's why we have to lift everything to the category of diagrams.

\medskip

In order to push down the results we will obtain for diagrams to actual algebras, we will need the following technical lemma.

\begin{lemma} \label{lemma:isomOfHoms}
	Let $\g$ be a Lie algebra which is complete with respect to its canonical filtration, and such that $\g^{(n)}\coloneqq\g/F^\lie_n\g$ is finite dimensional for all $n\ge2$. Let $\widetilde{\g}\in\pro{\mathsf{fd\ nil.\ Lie}}$ be the diagram
	\[
	\widetilde{\g} = \cdots\longrightarrow \g^{(4)}\longrightarrow \g^{(3)}\longrightarrow \g^{(2)}\longrightarrow 0\ .
	\]
	Let $H:\mathcal{D}\to\mathsf{fd\ nil.\ Lie}$ be any object of $\pro{\mathsf{fd\ nil.\ Lie}}$. Then we have a natural isomorphism
	\[
	\hom_{\pro{\mathsf{fd\ nil.\ Lie}}}(\widetilde{\g},H) \cong \hom_{\mathsf{dgLie}}(\g,\lim H)\ .
	\]
\end{lemma}

\begin{proof}
	We have
	\[
	\hom_{\mathsf{dgLie}}(\g,\lim H) = \lim_{d\in\mathcal{D}}\hom_{\mathsf{dgLie}}(\g,H(d))
	\]
	with all $H(d)$ finite dimensional. Fix $d\in\mathcal{D}$, then we have the linear diagram
	\[
	\cdots\longleftarrow\hom_{\mathsf{dgLie}}(\g^{(4)},H(d))\longleftarrow\hom_{\mathsf{dgLie}}(\g^{(3)},H(d))\longleftarrow\hom_{\mathsf{dgLie}}(\g^{(2)},H(d))\longleftarrow0\ .
	\]
	Since all the maps in the diagram $\widetilde{\g}$ were surjective, all the maps in this diagram are injective. Let $N$ be the nilpotency degree of $H(d)$. Then every map $\g\to H(d)$ splits through $\g^{(n)}$ for all $n\ge N$, so that we have
	\[
	\hom_{\mathsf{dgLie}}(\g,H(d))\cong\hom_{\mathsf{dgLie}}(\g^{(n)},H(d))
	\]
	for all $n\ge N$. It follows that
	\[
	\hom_{\mathsf{dgLie}}(\g,H(d))\cong\colim{n}\hom_{\mathsf{dgLie}}(\g^{(n)},H(d))\ ,
	\]
	which concludes the proof.
\end{proof}

In particular, this is true for any $\g$ of the form $\g\coloneqq\widehat{\lie}(V)$, for $V$ a finite dimensional graded vector space, with any differential. The examples of main interest to us will be the pronilpotent Lie algebras
\[
\mc_n\coloneqq \widehat{\Omega}_\pi(sC_n^\vee)
\]
introduced in \cref{subsect:mcbullet}.

\begin{lemma} \label{lemma:mcnCofibrant}
	Let $n\ge0$. We have
	$$\colim{}\ \widetilde{\mc}_n^\vee = \bar_\iota(s^{-1}C_n)\ .$$
	In particular, $\widetilde{\mc}_n$ is cofibrant.
\end{lemma}

\begin{proof}
	The dual of the quotient
	$$\mc_n^{(k)} = \mc_n/F^\lie_k\mc_n$$
	of $\mc_n$ by the $k$th space of the canonical filtration is exactly the $k$th space of the coradical filtration of $\bar_\iota(s^{-1}C_n)$, and the projections
	$$\mc_n^{(k+1)}\longrightarrow\mc_n^{(k)}$$
	become the inclusions of the various spaces of the coradical filtration of $\bar_\iota(s^{-1}C_n)$. From this, the result follows.
\end{proof}

\begin{lemma} \label{lemma:mc1Cylinder}
	The diagram $\widetilde{\mc}_1$ is a cylinder object for $\widetilde{\mc}_0$. After passing to the limit, the splitting of the codiagonal map becomes
	$$\mc_0\sqcup\mc_0\stackrel{i}{\longrightarrow}\mc_1\stackrel{t}{\longrightarrow}\mc_0$$
	where the two maps are as follows. The first Lie algebra is the free complete Lie algebra generated by two Maurer--Cartan elements $\alpha_0$ and $\alpha_1$, while the second one is the free complete Lie algebra generated by two Maurer--Cartan elements $\beta_1,\beta_2$ and a gauge $\lambda$ from $\beta_1$ to $\beta_2$, and the last algebra $\mc_0$ is the free Lie algebra on a single Maurer--Cartan element $\alpha$. The first map is determined by $i(\alpha_j) = \beta_j$ for $j=1,2$, the second map is determined by $t(\beta_j) = \alpha$ for $j=1,2$ and $t(\lambda) = 0$.
\end{lemma}

\begin{proof}
	By \cref{lemma:mcnCofibrant}, we know that $\widetilde{\mc}_n$ corresponds to the conilpotent Lie coalgebra $\bar_\iota(s^{-1}C_n)$. Since $\bar_\iota$ is right adjoint, we have
	$$\bar_\iota(s^{-1}C_0)\times\bar_\iota(s^{-1}C_0)\cong\bar_\iota(s^{-1}C_0\times s^{-1}C_0)$$
	with the obvious differential in the category of conilpotent Lie coalgebras. The diagonal map
	$$C_0\stackrel{\Delta}{\longrightarrow}C_0\times C_0$$
	in the category of $\C_\infty$-algebras splits into
	$$C_0\stackrel{T}{\longrightarrow}C_1\stackrel{I}{\longrightarrow}C_0\times C_0$$
	where $T(1) = \omega_0 + \omega_1$, $I(\omega_0) = (1,0)$, $I(\omega_1) = (0,1)$, and $I(\omega_{01}) = (0,0)$. The map $T$ is a quasi-isomorphism as it is induced by the homotopy equivalence $\Delta^1\to\Delta^0$. Therefore, by a result analogous to \cref{lemma:BarOfQIisFQI}, it is sent to a filtered quasi-isomorphism $\bar_\iota(T)$ by the bar construction. The map $I$ is clearly surjective, and thus $\bar_\iota(I)$ is a fibration by \cite[Thm. 2.9(2)]{vallette14}. Notice that the mentioned theorem is stated for the twisting morphism $\kappa$, but it also holds for $\iota$ by repeating the same proof and using \cref{lemma:modelStrKappaIota}. It follows that $\bar_\iota(s^{-1}C_1)$ is a path object for $\bar_\iota(s^{-1}C_0)$, which is equivalent to the statement we wanted to prove. Recovering the exact structure of the maps after passing to the limit is straightforward.
\end{proof}

We are now set for the proof of the Goldman--Millson theorem for pronilpotent Lie algebras.

\begin{theorem}\label{thm:homotopical GM}
	Let $\g,\h$ be two pronilpotent dg Lie algebras, and let $\phi:\g\to\h$ be the dual of a weak equivalence of conilpotent Lie coalgebras. Then $\phi$ induces a bijection
	\[
	\MCbar(\g)\cong\MCbar(\h)\ .
	\]
\end{theorem}

\begin{proof}
	Let $G,H\in\pro{\mathsf{fd\ nil.\ Lie}}$ and
	\[
	\Phi\in\hom_{\pro{\mathsf{fd\ nil.\ Lie}}}(G,H)
	\]
	be a weak equivalence such that $\lim G = \g$, $\lim H = \h$, and such that $\Phi$ corresponds to $\phi$ after passing to the limits. By definition, the moduli space of Maurer--Cartan elements of a Lie algebra $\g$ is the quotient
	\[
	\MCbar(\g) = \MC(\g)/\sim_{\mathrm{gauge}}
	\]
	of the space of Maurer--Cartan elements of $\g$ by the gauge relation. As explained above, we can see Maurer--Cartans elements in $\g$ as morphisms $\mc_0\to\g$ and gauge equivalences are coded by morphisms $\mc_1\to\g$. That is to say, two Maurer--Cartan elements $x_0,x_1\in\MC(\g)$ are gauge equivalent if, and only if there exists a morphism $\mc_1\to\g$ making the following diagram commute.
	\begin{center}
		\begin{tikzpicture}
			\node (a) at (0,0){$\mc_0$};
			\node (b) at (1,-1.5){$\mc_1$};
			\node (c) at (3,-1.5){$\g$};
			\node (d) at (0,-3){$\mc_0$};
			
			\draw[->] (a) -- node[right]{$i_0$} (b);
			\draw[->] (d) -- node[right]{$i_1$} (b);
			\draw[->] (b) -- (c);
			\draw[->] (a) to[out=0,in=140] node[above right]{$x_0$} (c);
			\draw[->] (d) to[out=0,in=-140] node[below right]{$x_1$} (c);
		\end{tikzpicture}
	\end{center}
	But by \cref{lemma:isomOfHoms} and \cref{lemma:mc1Cylinder}, this is exactly the left homotopy relation on morphisms $\widetilde{\mc}_0\to G$. Therefore,
	\[
	\MCbar(\g)\cong\hom_{\pro{\mathsf{fd\ nil.\ Lie}}}(\widetilde{\mc}_0,G)/\sim_\ell = \hom_{\mathsf{Ho}(\pro{\mathsf{fd\ nil.\ Lie}})}(\widetilde{\mc}_0,G)\ ,
	\]
	where the last equality is given by the fact that all elements of $\pro{\mathsf{fd\ nil.\ Lie}}$ are fibrant, as well as the fact that $\widetilde{\mc}_0$ is cofibrant by \cref{lemma:mcnCofibrant}. The same thing is of course true for $\h$, and since $\phi:\g\to\h$ comes from a weak equivalence, it naturally induces a bijection
	\[
	\MCbar(\g)\cong\MCbar(\h)
	\]
	by \cref{thm:model categories}, as desired.
\end{proof}

Although this result is slightly weaker than the full Goldman--Millson theorem (it works only on some algebras, and we don't have all the morphisms we would like, cf. \cref{rem:not all filtered qi are duals}), it has the advantage of having a fully homotopical proof, which is good considering the homotopical flavor of the statement.

\section{Framings and the Dolgushev--Rogers theorem} \label{sect:DRThm}

The last section can be seen as the ``level $0$" of a higher theory, which we present in this section. We start with a short recollection on simplicial frames following \cite[Ch. 5]{hovey} and then proceed to apply this theory to our model category, showing how the whole cosimplicial Lie algebra $\mc_\bullet$ appears naturally in this context. We are then able to recover the Dolgushev--Rogers theorem of \cite{dolgushev15} in this context by arguments similar to the ones used for the Goldman--Millson theorem in the last section.

\subsection{Simplicial frames}

Fix a model category $(\cat,W,C,F)$ and endow the category $s\cat\coloneqq\cat^{\Delta^{op}}$ of simplicial objects in $\cat$ with the Reedy model structure, see e.g. \cite[Thm. 5.2.5]{hovey}. Given an object $c\in\cat$, there are two natural ways to define an object $X_\bullet\in s\cat$:
\begin{enumerate}
	\item one can either ask that $X_n = L_nX$ is the $n$-th latching object, in which case $\ell_\bullet c\coloneqq X_\bullet$ is given by $\ell_n c = c$,
	\item or one can impose $X_n = M_nX$, the $n$-th matching object, obtaining a simplicial object $r_\bullet c\coloneqq X_\bullet$ given by $r_n c = c\times\cdots\times c$, the $(n+1)$-fold product of $c$.
\end{enumerate}
There is a canonical natural transformation
$$\ell_\bullet\longrightarrow r_\bullet$$
given by the identity in degree $0$ and by the diagonal in higher degrees.

\begin{definition}
	A \emph{simplicial frame} $c_*$ on an object $c\in\cat$ is a factorization of the canonical map $\ell_\bullet c\to r_\bullet c$ into a weak equivalence followed by a fibration that is an isomorphism in degree $0$,
	$$\ell_\bullet c\longrightarrow c_*\longrightarrow r_\bullet c\ .$$
\end{definition}

\begin{proposition}\label{prop:framings}
	Suppose $c$ is a fibrant object in $\cat$ and let $c_*$ be a simplicial frame on $c$. Then the functor
	$$\hom_\cat(-,c_*):\cat^{op}\longrightarrow\ssets$$
	preserves fibrations, trivial fibrations, and weak equivalences between fibrant objects.
\end{proposition}

\begin{proof}
	The fact that the functor preserves fibrations and trivial fibrations is proven in \cite[Cor. 5.4.4(2)]{hovey}. The fact that it preserves weak equivalences between fibrant objects then follows by Ken Brown's lemma (e.g. \cite[Lemma 1.1.12]{hovey}).
\end{proof}

\begin{remark}
	Notice that $\hom_\cat(-,c_*)$ is a functor that has $\C^{op}$ as domain (which is endowed with the dual model structure). Therefore, the proposition above states that the map given by the pullback by a \emph{co}fibration is a fibration of simplicial sets, and so on.
\end{remark}

\subsection{The Deligne--Hinich--Getzler $\infty$-groupoid and the Dolgushev--Rogers theorem}

We apply the theory exposed above to our case. We take for $\cat$ the category of conilpotent Lie coalgebras and for $c\in\cat$ the Lie coalgebra $\bar_\iota(s^{-1}C_0)$.

\begin{lemma}
	The simplicial conilpotent Lie coalgebra $\bar_\iota(s^{-1}C_\bullet)$ is a simplicial frame on $\bar_\iota(s^{-1}C_0)$.
\end{lemma}

\begin{proof}
	At level $0$, both maps are the identity. For $n\ge0$ the maps are induced by the maps of $\C_\infty$-algebras
	\[
	C_0\stackrel{w}{\longrightarrow}C_n\stackrel{p}{\longrightarrow}C_0\times\cdots\times C_0
	\]
	explicitly given by
	\[
	w(1) = \omega_0+\cdots+\omega_n
	\]
	and
	\[
	p(\omega_i) = (0,\ldots,0,\underbrace{1}_{i},0,\ldots,0),\qquad\mathrm{and}\qquad p(\omega_I) = 0\quad\forall |I|\ge2\ .
	\]
	The map $w$ is the pullback by the unique map $\Delta^n\to\Delta^0$, which is a homotopy equivalence. Thus, $w$ is a quasi-isomorphism and we obtain a weak equivalence when we apply the bar construction to it by an analogue to \cref{lemma:BarOfQIisFQI}. Similarly, the map $p$ is surjective, and thus is sent to a fibration by the bar construction, again by \cite[Thm. 2.9(2)]{vallette14} as in the proof of \cref{lemma:mc1Cylinder}.
\end{proof}

This simple fact is enough to recover all the most important results about $\hom_{\dgl}(\mc_\bullet,\g)$.

\begin{theorem}\label{thm:homotopical DR}
	The functor
	$$\hom_{\dgl}(\mc_\bullet,-):\mathsf{pronil.\ Lie}\longrightarrow\ssets$$
	sends duals of injections of conilpotent Lie coalgebras to fibrations and duals of weak equivalences of conilpotent Lie coalgebras to weak equivalences. In particular, it has image in the full subcategory of Kan complexes.
\end{theorem}

\begin{proof}
	Let $\g$ be a pronilpotent dg Lie algebra. Then $\g = \lim G$ for some $G\in\pro{\mathsf{fd\ nil.\ Lie}}$. Let $C\coloneqq \colim{}\,G^\vee$, so that $C^\vee = \g$. Then
	\begin{align*}
		\hom_{\dgl}(\mc_\bullet,\g)\cong&\ \hom_{\pro{\mathsf{fd\ nil.\ Lie}}}(\widetilde{\mc}_\bullet,G)\\
		\cong&\ \hom_{\mathsf{conil.\ coLie}}(C,\bar_\iota(s^{-1}C_\bullet))\ ,
	\end{align*}
	where the first isomorphism is given by \cref{lemma:mcnCofibrant}, and the second one by the equivalence of categories of \cref{cor:equivCats}. By \cref{thm:ValletteModelStr}, we know that $C$ is always cofibrant (and thus fibrant in the opposite category), and that $\bar_\iota(s^{-1}C_0)$ is fibrant, since it is quasi-free. \cref{prop:framings} concludes the proof.
\end{proof}

The statement that the functor sends duals of injections to fibrations should be seen analogous to \cref{thm:Getzler}, while the assertion that duals of weak equivalences of coalgebras are sent to weak equivalences is a weaker version of \cref{thm:DR}, in the same way that \cref{thm:homotopical GM} is a weaker version of \cref{thm:GM}.

\subsection{Relations with the BFMT model structure on Lie algebras}\label{subsect:relations with BFMT}

In the article \cite{buijs16}, Buijs--F\'elix--Murillo--Tanr\'e introduced a model structure on a slightly larger category: the category of all complete Lie algebras and filtered morphisms. They have the declared goal of developing a way to do rational homotopy theory for all spaces using Lie algebra models. We summarize their results and compare them with the present work.

\medskip

Let $\g$ be a complete Lie algebra, and let $x\in\MC(\g)$. Then one can use $x$ to twist the differential to get
\[
d^x\coloneqq d + \ad_x\ .
\]
This new operator $d^x$ is such that $(\g,d^x)$ is once again a Lie algebra.

\begin{definition}
	The \emph{component of $\g$ at $x$} is the Lie algebra obtained by truncating $(\g,d^x)$ in positive degree and taking only $\ker d^x$ in degree $0$, i.e. it is the cochain complex
	\[
	\cdots\xrightarrow{d^x}\g^{-2}\xrightarrow{d^x}\g^{-1}\xrightarrow{d^x}\ker d^x\longrightarrow0
	\]
	endowed with the Lie algebra structure inherited from $\g$.
\end{definition}

One defines the following three classes of maps:
\begin{itemize}
	\item A filtered morphism is a \emph{fibration} if it is surjective in non-negative degrees.
	\item A filtered morphism $\phi:\g\to\h$ is a \emph{weak equivalence} if $\MCbar(\phi)$ is a bijection and $\phi:\g^x\to\h^{\phi(x)}$ is a quasi-isomorphism for all $x\in\MC(\g)$.
	\item A filtered morphism is a \emph{cofibration} if it has the left lifting property with respect to all trivial cofibrations.
\end{itemize}

\begin{theorem}[{\cite[Thm. 3.1]{buijs16}}]
	These three classes of maps define a model structure on the category of complete Lie algebras and filtered morphisms. We call this model structure the BFMT model structure on complete Lie algebras.
\end{theorem}

The weak equivalences and fibrations of the model structure considered in the present paper are contained in the weak equivalences and fibrations of the BFMT model structure, as is proven in \cite[Sect. 6]{buijs16}.

\medskip

One also has a Quillen pair
\[
\adjunction{\mathfrak{L}}{\ssets}{\cdgl}{\hom_{\dgl}(\mc_\bullet,-)}\ ,
\]
where the functor $\mathfrak{L}$ is obtained by sending the standard $n$-simplex $\Delta^n$ to $\mc_n$ and then applying the Yoneda lemma to extend it to all simplicial sets. This is proven in \cite[Cor. 3.6]{buijs16}. In particular, both functors preserve weak equivalences. Then one proves:

\begin{proposition}[{\cite[Prop. 3.8]{buijs16}}]\label{prop:filterd qi are we in BFMT}
	The class of filtered quasi-isomorphisms is contained in the class of weak equivalences.
\end{proposition}

Thus, we can give the following alternative proof of the Dolgushev--Rogers theorem.

\begin{proof}[Proof of \cref{thm:DR}]
	This is a direct consequence of the two results above and the fact that if $\g$ is a complete Lie algebra, then there is a canonical homotopy equivalence of simplicial sets
	\[
	\MC_\bullet(\g)\simeq\hom_{\dgl}(\mc_\bullet,\g)
	\]
	as was proven in \cite[Cor. 5.3]{rn17cosimplicial} (see also \cite{buijs17}).
\end{proof}

In fact, our argument using framings holds in this model category, too. Indeed, the natural sequence of maps
\[
\mc_0\sqcup\ldots\sqcup\mc_0\longrightarrow\mc_n\longrightarrow\mc_0\ ,
\]
which is dual to the one showing that $\bar_\iota(s^{-1}C_\bullet)$ is a simplicial frame on $\bar_\iota(s^{-1}C_0)$, exhibits $\mc_\bullet$ as a cosimplicial frame on $\mc_0$ in the category of complete Lie algebras with the BFMT model structure: the second map is a weak equivalence, since it comes from a weak equivalence in the model structure of \cref{sect:ValletteModelStr}, while the first one is a cofibration by \cite[Thm. 4.2]{buijs16}. This gives yet another alternative proof of \cref{thm:DR}. However, one should remark that the proof of \cref{prop:filterd qi are we in BFMT} in \emph{loc. cit.} relies on the Dolgushev--Rogers theorem itself, and thus their proof of \cref{thm:DR} is --- in a sense --- not self-contained. On the other hand, \cref{thm:homotopical GM} and \cref{thm:homotopical DR} are  proved using only our homotopical approach, but they do not recover the full strength of the Dolgushev--Rogers theorem.

\bibliographystyle{alpha}
\bibliography{A_model_structure_for_the_Goldman-Millson_theorem}

\end{document}